\documentclass[letterpaper, 10 pt, conference]{ieeeconf}

\IEEEoverridecommandlockouts
\makeatletter

\let\proof\@undefined
\let\endproof\@undefined
\makeatother

\usepackage{makeidx}         
\usepackage{graphicx}        
\usepackage{multicol}        
\usepackage[bottom]{footmisc}
\usepackage{latexsym}

\usepackage[caption=false,font=footnotesize]{subfig}
\usepackage{cite,color,comment,xspace}
\usepackage{amsfonts}
\usepackage{amsmath}
\usepackage{amssymb}
\usepackage{amsthm}
\usepackage{algorithm}
\usepackage{algorithmic}
\usepackage{url,cite}
\usepackage{times}
\usepackage{enumerate}
\usepackage{epstopdf}
\usepackage{epsfig}
\usepackage{array}
\usepackage{fixltx2e}

\graphicspath{{fig/}{../fig/}}

\newtheorem{theorem}{Theorem}[section]

\newtheorem{definition}[theorem]{Definition}
\newtheorem{lemma}[theorem]{Lemma}
\newtheorem{remark}[theorem]{Remark}

\newtheorem{problem}[theorem]{Problem}

\newcommand{\B}{\mathcal{B}}
\newcommand{\N}{\mathcal{N}}
\newcommand{\R}{\mathcal{R}}
\newcommand{\Prod}{\mathcal{P}}
\newcommand{\T}{\mathcal{T}}

\newcommand{\bq}{\textbf{q}}
\newcommand{\bp}{\textbf{p}}
\newcommand{\bP}{\textbf{P}}
\newcommand{\bo}{\textbf{o}}
\newcommand{\rh}{\mathtt{RH}}

\newcommand{\Land}{\wedge}
\newcommand{\Lor}{\vee}
\newcommand{\Next}{\mathsf{X}\, }
\newcommand{\Always}{\mathsf{G} \,}
\newcommand{\Event}{\mathsf{F} \,}
\newcommand{\Until}{\, \mathsf{U} \,}

\newcommand{\n}{\nonumber\\}
\newcommand{\st}{\,|\,}

\newcommand{\be}{\begin{equation}}
\newcommand{\ee}{\end{equation}}
\newcommand{\ben}{\begin{equation*}}
\newcommand{\een}{\end{equation*}}
\newcommand{\bea}{\begin{eqnarray}}
\newcommand{\eea}{\end{eqnarray}}
\newcommand{\bean}{\begin{eqnarray*}}
\newcommand{\eean}{\end{eqnarray*}}
\newcommand{\ba}{\begin{array}}
\newcommand{\ea}{\end{array}}
\newcommand{\leftm}{\left[\begin{array}}
\newcommand{\rightm}{\end{array}\right]}

\newcommand{\ie}{{\it i.e., }}
\newcommand{\eg}{{\it e.g., }}

\definecolor{forestgreen}{rgb}{0,0.6,0.6}
\definecolor{brown}{rgb}{0.6,0.2,0.2}

\DeclareMathOperator*{\argmax}{arg\,max}

\newcommand\oprocendsymbol{\hbox{$\bullet$}}
\newcommand\oprocend{\relax\ifmmode\else\unskip\hfill\fi\oprocendsymbol}

\makeindex             

\title{\LARGE \bf
Receding Horizon Temporal Logic Control\\ for Finite Deterministic Systems}
\author{Xuchu Ding, Mircea Lazar and Calin Belta
\thanks{X.C. Ding and C. Belta are with the Dept. of Mechanical Engineering at Boston University, Brookline, MA 02446, USA.   Email: {\tt\small \{xcding,cbelta\}@bu.edu}. Mircea Lazar is with Dept. of Electrical Engineering at Eindhoven University of Technology, Eindhoven, The Netherlands. Email: {\tt\small m.lazar@tue.nl}. This work is partially supported by the ONR-MURI Award N00014-09-1051 at Boston University and the STW Veni Grant 10230 at Eindhoven University of Technology.}
}
\begin{document}
\maketitle \thispagestyle{empty} \pagestyle{empty}

\begin{abstract}
This paper considers receding horizon control of finite deterministic systems, which must satisfy a high level, rich specification expressed as a linear temporal logic formula. Under the assumption that time-varying rewards are associated with states of the system and they can be observed in real-time, the control objective is to maximize the collected reward while satisfying the high level task specification. In order to properly react to the changing rewards, a controller synthesis framework inspired by model predictive control is proposed, where the rewards are locally optimized at each time-step over a finite horizon, and the immediate optimal control is applied. By enforcing appropriate constraints, the infinite trajectory produced by the controller is guaranteed to satisfy the desired temporal logic formula. Simulation results demonstrate the effectiveness of the approach.
\end{abstract}

\section{Introduction}
\label{sec:intro}

This paper considers the problem of controlling a deterministic discrete-time system with a finite state-space, which is also referred to as a finite transition system. Such systems can be effectively used to capture behaviors of more complex dynamical systems, and as a result, greatly reduce the complexity of control design.

A finite transition system can be constructed from a continuous system via an ``abstraction'' process.  For example, for an autonomous robotic vehicle moving in an environment, the motion of the vehicle can be abstracted to a finite system through a partition of the environment.   The set of states can be seen as a set of labels for the regions in the partition, and each transition corresponds to a controller driving the vehicle between two adjacent regions.  By partitioning the environment into simplicial, rectangular or polyhedral regions, continuous feedback controllers that drive a robotic system from any point inside a region to a desired facet of an adjacent region have been developed for linear \cite{KB-TAC08-LTLCon}, multi-affine \cite{HabColSchup06}, piecewise-affine \cite{desai2002controlling, habets2007control, Wongpiromsarn:2009p19}, and non-holonomic (unicycle) \cite{belta2005discrete, lindemann2007real} dynamical models.  By relating the initial continuous dynamical system and the abstract discrete finite system with simulation or bisimulation relations \cite{Milner89}, the abstraction process allows one to solve a control problem for the more complex continuous system with the ``equivalent'' abstract system.

It has been proposed by several authors \cite{KB-TAC08-LTLCon,Hadas-ICRA07,Karaman_mu_09,Loizou04,Wongpiromsarn:2009p19} to use temporal logics, such as linear temporal logic (LTL) and computation tree logic (CTL) \cite{Clarke99}, as specification languages for finite transition systems due to their well defined syntax and semantics.  These logics can be easily used to specify complex behavior, and in particular with LTL, persistent mission task such as ``pick up items at the region $\mathtt{pickup}$, and then drop them off at the region $\mathtt{dropoff}$, infinitely often, while always avoiding $\mathtt{unsafe}$ regions''.  The applications of these temporal logics in computer science in the area of model checking~\cite{Clarke99} and temporal logic game~\cite{Piterman-2006} has resulted in off-the-shelf tools and algorithms that can be readily adapted to synthesize provably correct control strategies \cite{KB-TAC08-LTLCon,Hadas-ICRA07,Karaman_mu_09,Loizou04,Wongpiromsarn:2009p19}.
%

While the works mentioned above address the temporal logic controller synthesis problem, several problems and questions remain to be answered.  In particular, the problem of combining temporal logic controller synthesis with optimality with respect to a suitable cost function remains to be solved.  This problem becomes even more difficult if the optimization problem depends on time-varying parameters, \eg dynamic events that occur during the operation of the plant.  For traditional control problems (without temporal logic constraints) and dynamical systems, this problem can be effectively addressed using a model predictive control (MPC) approach (see \eg \cite{rawlings2009}), which has reached a mature level in both academia and industry, with many successful implementations. The basic MPC set-up consists of the following sequence of steps: at each time instant, a cost function of the current state is optimized over a finite horizon, only the first element of the optimal finite sequence of controls is applied and the whole process is repeated at the next time instant for the new measured state. Thus, MPC is also referred to as receding horizon control. Since the finite horizon optimization problem is solved repeatedly at each time instant, real-time dynamical events can be effectively managed.

However, it is not yet well-understood how to combine a receding horizon control approach with a provably correct control strategy satisfying a temporal logic formula.  The aim of this paper is to address this issue for a specific system set-up (deterministic systems on a finite state-space) and problem formulation (dynamic optimization of rewards). More specifically, the role of the receding horizon controller is to maximize over a finite horizon the accumulated rewards associated with states of the system, under the assumption that the rewards change dynamically with time and they can only be observed in real-time.  The rewards model dynamical events that can be triggered in real-time, which is an often used model in coverage control literature \cite{li2006cooperative}.

The key challenge in this controller synthesis framework is to ensure correctness of the produced infinite trajectory and recursive feasibility of the optimization problem solved at each time-step.  For a constrained MPC optimization problem, which is solved recursively on-line, feasible at all times or recursively feasible means that if the optimization problem is feasible (has a solution) for the initial state at initial time, then it remains feasible for all future time instants, when it will be solved with a different initial condition resulting from the generated closed-loop trajectory.  A proof that the proposed receding horizon control framework satisfies both properties is provided. Similar to standard MPC, where certain \emph{terminal} constraints must be enforced in the optimization problem in order to guarantee certain properties for the system (\eg stability), the correctness of produced trajectory and recursive feasibility are also ensured via a set of suitable constraints.

This work can be seen as an extension and generalization of the set-up presented in \cite{ding2010receding}, where a similar control objective was tackled. In \cite{ding2010receding} an optimization based controller was designed, which consists of repeatedly solving a finite horizon optimal control problem every $N$ steps and implementing the complete sequence of control actions. This procedure is more close to finite-horizon optimal control than true receding horizon control and its main drawback comes from the inability of reacting to dynamical events (\ie rewards) triggered or varying during the execution of the finite trajectory. This paper removes this limitation by attaining a truly receding horizon controller for deterministic systems on a finite state-space. Another related work is \cite{Wongpiromsarn:2009p19}, where a provably correct control strategy was obtained for large scale systems by dividing the control synthesis problem into smaller sub-problems in a receding horizon like manner.  However, in \cite{Wongpiromsarn:2009p19} dynamical events were addressed differently and the specification language was restricted to a fragment of LTL, whereas in this paper full LTL expressivity is allowed.

\section{Problem Formulation and Approach}
\label{sec:problem}
In this paper, we consider a discrete-time system with a finite state space, \ie the system evolves on a graph.   Each vertex of the graph produces an output, which is a set of observations.   Such a system can be described by a finite deterministic transition system, which can be formally defined as follows.
\begin{definition}[Finite Deterministic Transition System]\label{def:tran_sys}
A finite (weighted) deterministic transition system (DTS) is a tuple $\T=(Q,q_0,\Delta,\omega,\Pi,h)$, where
\begin{itemize}
\item $Q$ is a finite set of states;
\item $q_0\in Q$ is the initial state;
\item $\Delta \subseteq Q\times Q$ is the set of transitions;
\item $\omega:\Delta\to \mathbb R^{+}$ is a weight function that assigns positive values to all transitions;
\item $\Pi$ is a set of observations; and
\item $h : Q \rightarrow 2^\Pi$ is the observation map.
\end{itemize}
For convenience of notation, we denote $q\to_{\T} q'$ if $(q,q')\in\Delta$.   We assume $\mathcal T$ to be non-blocking, \ie for each $q\in Q$, there exists $q'\in Q$ such that $q\to_{\T} q'$ (such a system is also called a Kripke structure \cite{browne1988characterizing}).
A {\em trajectory} of a DTS is an infinite sequence $\bq=q_{0}q_{1}...$ where $q_{k}\to_{\T} q_{k+1}$ for all $k\geq 0$.  A trajectory $\bq$ generates an output trajectory $\bo=o_{0}o_{1}...$, where $o_{k}=h(q_{k})$ for all $k\geq 0$.   
\end{definition}

Note the absence of the control inputs in the definition of $\T$.  This is because $\T$ is deterministic, and one can choose an available transitions at a state.   In other words, each transition $(q,q')$ corresponds to a unique control input at state $q$.  This also implies that a trajectory $\bq=q_{0}q_{1}\ldots$ can be used as a control strategy for $\T$, by simply applying the transitions $(q_{0},q_{1}), (q_{1},q_{2})$, and so on.  An example of a DTS is shown in Fig. \ref{fig:Texample}.
\begin{figure}[!ht]
	\centering
	\includegraphics[scale=.35]{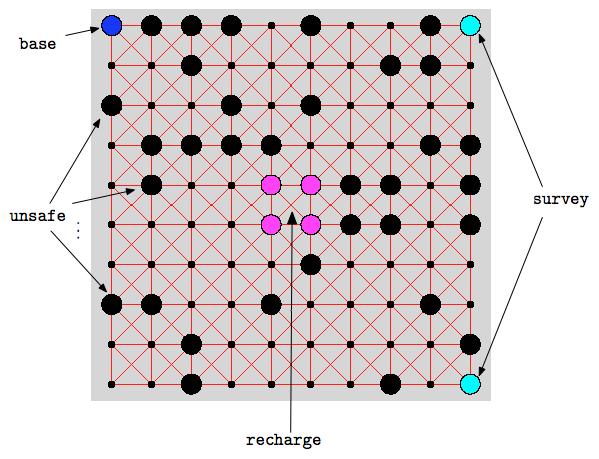}
	\caption{An example of a finite DTS $\T$ as defined in Def. \ref{def:tran_sys}.  In this example, $\mathcal T$ has $100$ states, which are at the vertices of a rectangular grid with cell size $10$.  We define the weight function $\omega$ to be the Euclidean distance between vertices, and there is a transition between two vertices if the Euclidean distance between them is less than $15$.  The set of observations is $\Pi=\{\mathtt{base}, \mathtt{survey}, \mathtt{recharge}, \mathtt{unsafe}\}$. States with no observation are shown with smaller vertices.  
	}
\label{fig:Texample}
\end{figure}

The goal of this paper is to synthesize trajectories $\bq$ of $\T$ satisfying a behavioral specification, given as a linear temporal logic formula over $\Pi$.    An LTL formula over $\Pi$ is interpreted over an (infinite) sequence $\bo=o_{0}o_{1}\ldots$, where $o_{k}\subseteq \Pi$ for all $k\geq 0$.  We say $\bq$ satisfies an LTL formula $\phi$ if it generates an output trajectory $\bo$ satisfying $\phi$. A detailed description of the syntax and semantics of LTL is beyond the scope of this paper and can be found in \cite{Clarke99}.  Roughly, an LTL formula is build up from the observations in $\Pi$, Boolean operators $\neg$ (negation), $\Lor$ (disjunction), $\Land$ (conjunction), $\longrightarrow$ (implication), and temporal operators $\Next$ (next), $\Until$ (until), $\Event$ (eventually), $\Always$ (always).   For example, the following task command in natural language: \emph{``Reach
a $\mathtt{survey}$ location infinitely often, and always avoid $\mathtt{unsafe}$ states''} can be translated to the LTL formula: $\phi:=\Always \Event \mathtt{survey} \Land \Always \neg \mathtt{unsafe}$.

The system is assumed to operate in an environment with dynamical events.  In this paper, these events are modelled by a reward process $\R:Q\times \mathbb N\rightarrow \mathbb R^{+}$, \ie the reward associated with state $q\in Q$ at time $k$ is $\R(q,k)$.  Note that rewards are associated with states in $Q$ in a time varying fashion.  We do not make any assumptions on the dynamics governing the rewards, but we make the natural assumption that, at time $k$, the system can only observe the rewards in a neighborhood $\N(q, k)\subseteq Q$ of the current state $q$.   In this paper, we assume that the reward process $\R$ is unknown and reward values must be observed and acted upon in real-time.  The problem in the case when knowledge of $\R$ is given a-priori is interesting, but will be addressed in future research.


%

The problem considered in this paper is formally stated next.
\begin{problem}\label{problem:main}
Given a transition system $\mathcal T$ and an LTL formula $\phi$ over the set of observations of $\T$, design a controller that maximizes the collected reward locally, while it ensures that the produced infinite trajectory satisfies $\phi$.
\end{problem}

Since the rewards are time-varying and can only be observed around the current state, inspirations from the area of MPC are drawn (see, e.g. \cite{rawlings2009}) with the aim of synthesizing a controller such that the rewards are maximized in a receding horizon fashion.
At time $k$ with state $q_{k}$, the controller generates a finite trajectory $q_{k+1}q_{k+2}\ldots q_{k+N}$ by solving an on-line optimization problem maximizing the collected rewards over a horizon $N$, and the system implements the immediate control action $(q_{k}, q_{k+1})$.   This process is then repeated at time $k+1$ and state $q_{k+1}$.  

In order to guarantee the satisfaction condition for the LTL formula $\phi$, the proposed approach is based on the construction of an automaton that captures all satisfying trajectories of $\T$. This automaton also induces a Lyapunov-like function that can be used to enforce that the trajectory of the system  satisfies the desired formula. These steps are formally described in detail in Sec.~\ref{sec:prelim}. The aforementioned function will be utilized to guarantee recursive feasibility of the developed receding horizon controller, which in turn will yield that the synthesized infinite trajectory satisfies $\phi$. The controller synthesis method is presented in Sec.~\ref{sec:control}.

\section{A tool for enforcing\\ the B\"uchi acceptance condition}
\label{sec:prelim}



In this section, we review the definition of B\"uchi automata and describe the construction of a function that enforces the satisfaction of a B\"uchi acceptance condition for the trajectories of a DTS.

%

\begin{definition}[B\"{u}chi Automaton]\label{def:omega_aut}
A (nondeterministic) B\"{u}chi automaton is a tuple $\B=(S_{\B},S_{\B0},\Sigma,\delta,F_{\B})$, where
\begin{itemize}
\item $S_{\B}$ is a finite set of states;
\item $S_{\B0}\subseteq S_{\B}$ is the set of initial states;
\item $\Sigma$ is the input alphabet;
\item $\delta: S_{\B}\times \Sigma \rightarrow 2^{S_{\B}} $ is the transition function;
\item $F_{\B}\subseteq S$ is the set of accepting states.
\end{itemize}
We denote $s\overset{\sigma}{\to}_{\B}s'$ if $s'\in \delta(s,\sigma)$.   An infinite sequence $\sigma_{0}\sigma_{1}\ldots$ over $\Sigma$ generates trajectories $s_{0}s_{1}\ldots$ where $s_{0}\in S_{\B0}$ and $s_{k}\overset{\sigma_{k}}{\to}_{\B}s_{k+1}$ for all $k\geq 0$.  $\B$ accepts an infinite sequence over $\Sigma$ if it generates at least one trajectory on $\B$, which intersects the set $F_{\mathcal B}$ infinitely many times.
\end{definition}
For any LTL formula $\phi$ over $\Pi$, one can construct a B\"{u}chi automaton with input alphabet $\Sigma= 2^{\Pi}$ accepting all and only sequences over $2^{\Pi}$ that satisfy $\phi$ \cite{Clarke99}.    We refer readers to \cite{gastin2001fast} for efficient algorithms and implementations to translate an LTL formula over $\Pi$ to a corresponding B\"{u}chi automaton $\mathcal B$.

\begin{definition}[Weighted Product Automaton]\label{def:PA}
 Given a weighted DTS $\T=(Q,q_0,\Delta,\omega,\Pi,h)$ and a B\"uchi automaton $\B = (S_{\B},S_{\B 0},2^{\Pi},\delta_\B,F_{\B})$, their product automaton, denoted by $\Prod=\T \times \B$, is a tuple $\Prod=(S_\Prod,S_{\Prod0},\Delta_\Prod, \omega_{\Prod},F_\Prod)$
where
\begin{itemize}
 \item $S_\Prod = Q \times S_{\B}$;
 \item $S_{\Prod0} = \{q_0\} \times S_{\B 0}$;
 \item $\Delta_{\Prod}\subseteq S_{\Prod}\times S_{\Prod}$ is the set of transitions, defined by: $\left((q,s),(q',s')\right)\in \Delta_{\Prod}$ iff $q\to_{\T} q'$ and $s\overset{h(q)}{\longrightarrow}_{\B}s'$;
 \item $\omega_{\Prod}:\Delta_{\Prod}\to \mathbb R^{+}$ is the weight function defined by: $\omega_{\Prod}\left((q,s),(q',s')\right)=\omega\left((q,q')\right)$
 \item $F_\Prod = Q \times F_{\B}$.
\end{itemize}
We denote $(q,s) \to_{\Prod} (q',s')$ if $((q,s), (q',s'))\in \Delta_{\Prod}$.
A trajectory ${\rm \bp}=(q_{0}, s_{0})(q_{1},s_{1})\ldots$ of $\Prod$ is an infinite sequence such that $(q_{0}, s_{0})\in S_{\Prod 0}$ and $(q_{k},s_{k}) \to_{\Prod} (q_{k+1},s_{k+1})$ for all $k\geq 0$.   Trajectory $\bp$ is called accepting if and only if it intersects  $F_{\mathcal P}$ infinitely many times.
\end{definition}
 We define the projection $\gamma_{\T}$ of $\bp$ onto $\mathcal T$ as simply removing the automaton states, \ie
\be
\label{def:projection}
\gamma_{\mathcal T}(\bp)=\bq=q_{0}q_{1}\ldots, \textrm { if } \bp=(q_{0},s_{0})(q_{1},s_{1})\ldots .
\ee
We also use the projection operator $\gamma_{\T}$ for finite trajectories (subsequences of $\bp$).  
Note that a trajectory $\bp$ on $\Prod$ is uniquely projected to a trajectory $\gamma_{\T}(\bp)$ on $\T$.
By the construction of $\Prod$ from $\T$ and $\B$, $\bp$ is accepted if and only if $\bq=\gamma_{\mathcal T}(\bp)$ satisfies the LTL formula corresponding to $\B$ \cite{Clarke99}. 

\begin{figure*}[!ht]
	\centering
	\label{fig:seq}
	\subfloat[]{	\includegraphics[scale=.4]{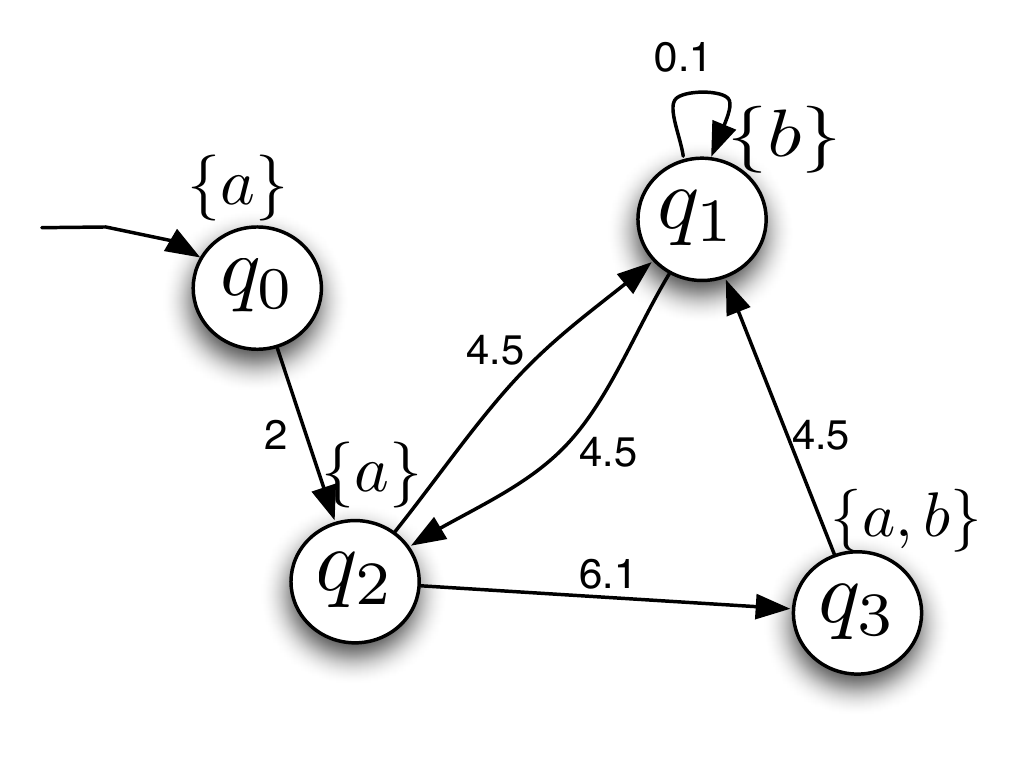}}
	\subfloat[]{	\includegraphics[scale=.4]{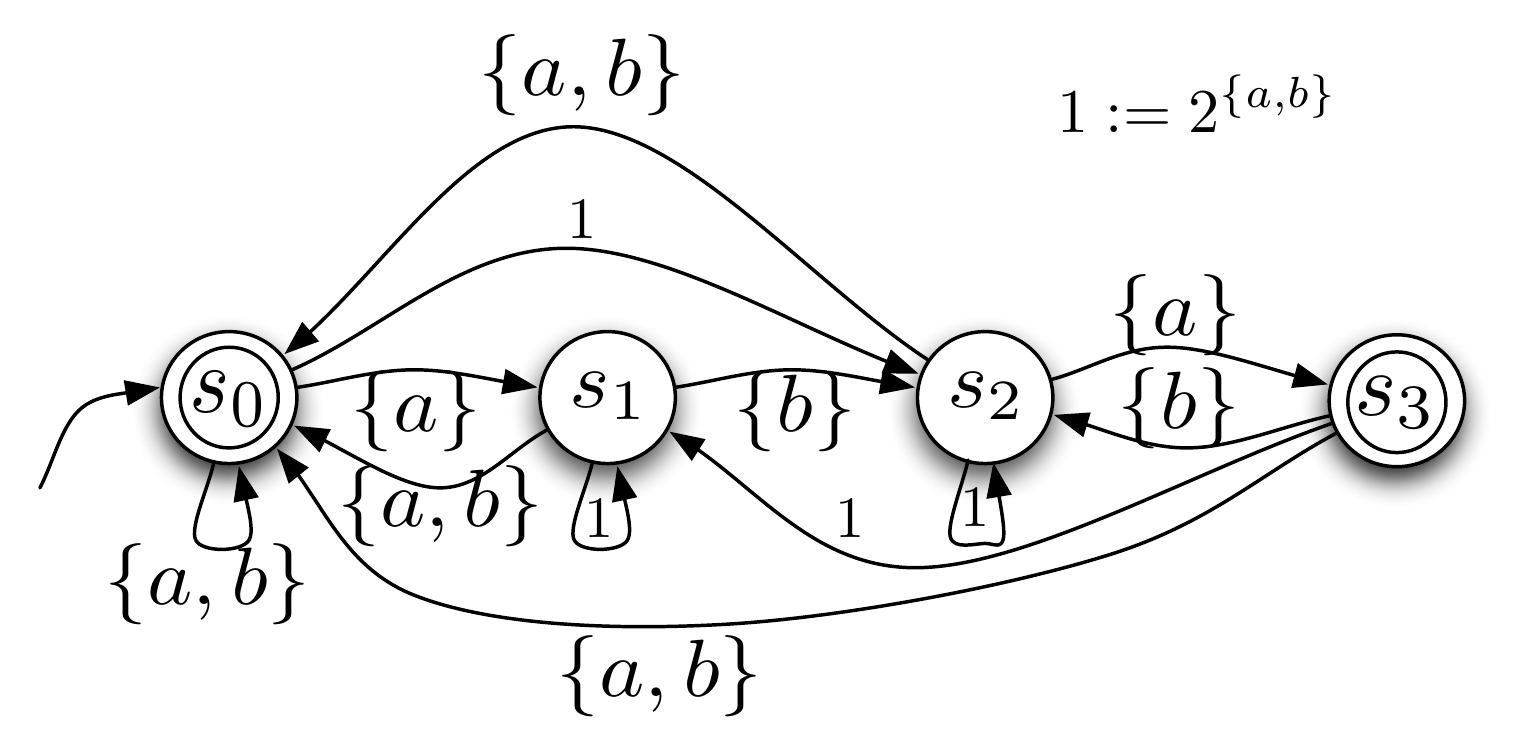}}
	\subfloat[]{	\includegraphics[scale=.4]{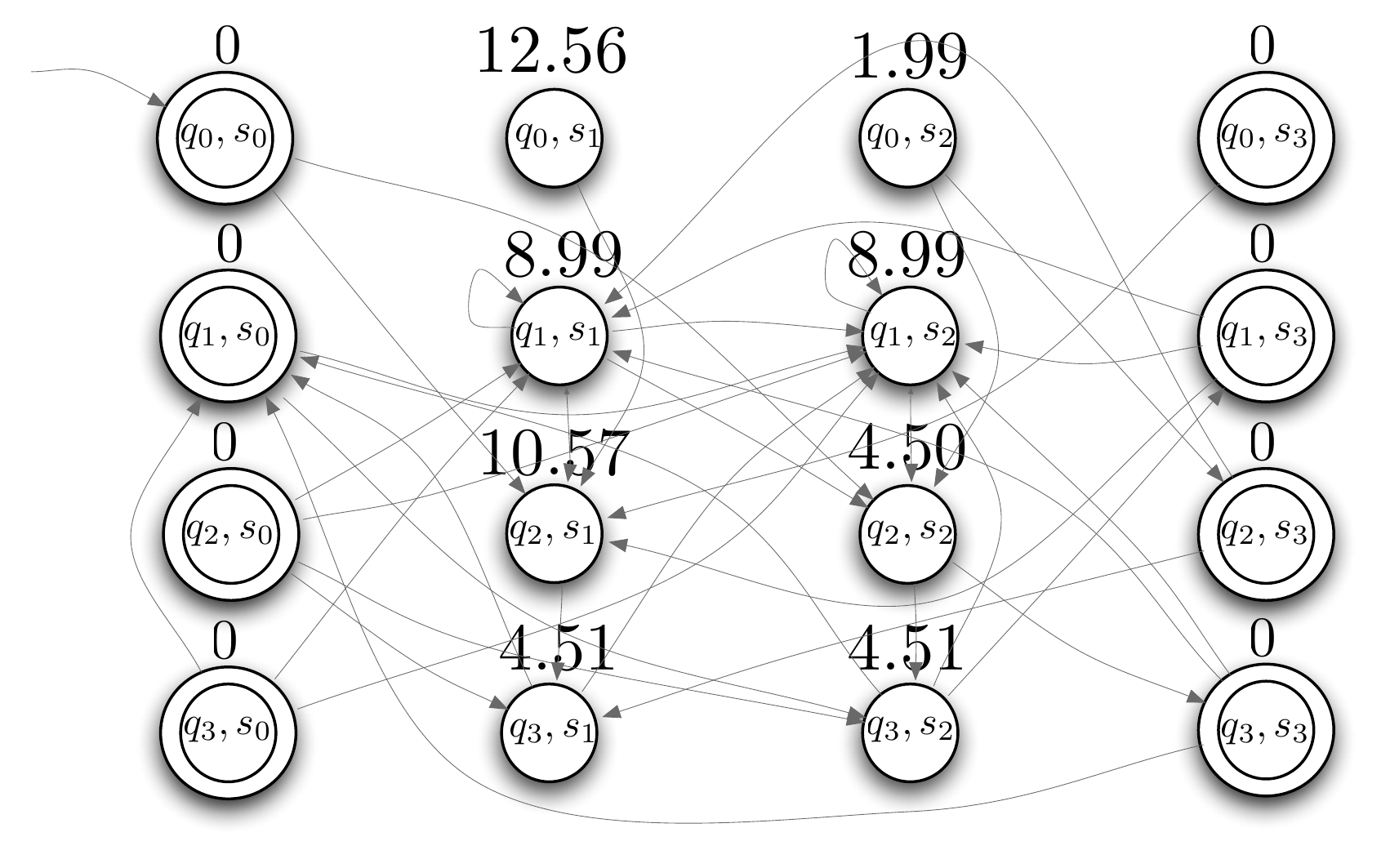}}
	\caption{The construction of the product automaton and the energy function on its states.  In this example, the set of observations is $\Pi=\{a,b\}$.  The initial states are indicated by incoming arrows.  The accepting states are marked by double-strokes.  {\bf (a)}:  A weighted DTS $\T$.  The label atop each state indicates the set of associated observations. (\ie $\{a,b\}$ means both $a$ and $b$ are observed).  The labels on the transitions indicate the weights.   {\bf (b)}: The B\"{u}chi automaton $\B$ corresponds to LTL formula $\Always (\Event (a \Land \Event b))$, translated by the tool LTL2BA~\cite{gastin2001fast}.  {\bf (c)}: The product automaton $\Prod=\T\times\B$ constructed with Def. \ref{def:PA} (the weights are inherited from $\T$ and not shown).  The number above a state $p\in S_{\Prod}$ is the energy function $V(p)$.  Note that in this example, the set $F^{\star}_{\Prod}=F_{\Prod}$ , thus $V(p)$ is the graph distance from $p$ to any accepting states.}
\label{fig:example}
\end{figure*}

In \cite{ding2010receding}, we introduced a real positive function $V$ on the states of the product automaton $\Prod$ that uses the weights $\omega_\Prod$ to enforce the acceptance condition of the automaton. Conceptually, this function resembles a Lyapunov, or energy function. While in Lyapunov theory energy functions are used to enforce that the trajectories of a dynamical system converge to an equilibrium, this ``energy" function enforces that the trajectories of $\T$ satisfy the acceptance condition of a B\"uchi automaton.  

To define the energy function, we first denote a set $A\subseteq S_{\mathcal P}$ to be {\em self-reachable} if and only if all states in $A$ can reach at least one state in $A$.
\begin{definition}[Energy function of a state in $\Prod$]
We define $F_{\mathcal P}^{\star}$ to be the largest self-reachable subset of $F_{\mathcal P}$.   The energy function $V(p)$, $p\in S_{\Prod}$ is defined as the graph distance of $p$ to the set $F^{\star}_{\mathcal P}$, \ie the accumulated weight of the shortest path from $p$ to any states in $F^{\star}_{\mathcal P}$.
\end{definition}
Fig. \ref{fig:example} shows an example of $\T$, $\B$, and their product $\Prod$, as well as the induced energy function defined on states of $\Prod$.
In \cite{ding2010receding}, we showed the following properties for $V$.
\begin{theorem}[Properties of the energy function] $V$ satisfies the following:
 \begin{enumerate}
 \item If a trajectory $\bp$ on $\Prod$ is accepting, then it cannot contain a state $p$ where $V(p)=\infty$.
 \item All accepting states in an accepting trajectory $\bp$ are in the set $F^{\star}_{\Prod}$ and have energy equal to $0$; all accepting states that are not in $F^{\star}_{\Prod}$ have energy equal to $\infty$.
 \item For each state $p\in S_{\Prod}$, if $V(p)>0$ and $V(p)\neq \infty$, then there exists a state $p'$ where $p\to_{\Prod}p'$ such that $V(p')<V(p)$.
\end{enumerate}
\label{thm:Vfacts}
\end{theorem}
We see from the above facts that $V(p)$ resembles an energy-like function, which justifies the name we use.   We refer to the value of $V(p)$ at a state $p\in S_{\Prod}$ as the ``energy of the state''.  Note that satisfying the LTL formula is equivalent to reaching states where $V(p)=0$ for infinitely many times.  Therefore, for each state $p\in S_{\Prod}$, $V(p)$ provides a measure of progress towards satisfying the LTL formula.
An algorithm generating $V(p)$ for an arbitrary product automaton can be found in \cite{ding2010receding}.

\section{Main results:\\Receding Horizon Controller Design}
\label{sec:control}

In this section, we present a solution to Prob. \ref{problem:main}.   The central component of our control design is a state-feedback controller operating on the product automaton that optimizes finite trajectories over a pre-determined, fixed horizon $N$, subject to certain constraints.  These constraints ensure that the energy of states on the product automaton decreases in finite time, thus guaranteeing that progress is made towards the satisfaction of the LTL formula.  Note that the proposed controller does not enforce the energy to decrease at each time-step, but rather that it eventually decreases.
The finite trajectory returned by the receding horizon controller is projected onto $\T$, the controller applies the first transition, and this process is repeated again at the next time-step.

In this section, we first describe the receding horizon controller and show that it is feasible (a solution exists) at all time-steps $k\in\mathbb N$.  
Then, we present the general control algorithm and show that it always produces (infinite) trajectories satisfying the given LTL formula.

\subsection{Receding horizon controller}
\label{subsec:rhc}
In order to explain the working principle of the controller, we first define a finite \emph{predicted trajectory} on $\Prod$ at time $k$.    Denote the current state at time $k$ as $p_{k}$.  A predicted trajectory of horizon $N$ at time $k$ is a finite sequence $\bp_{k}:=p_{1|k}\ldots p_{N|k}$, where $p_{i|k}\in S_{\Prod}$ for all $i=1,\ldots,N$,  $p_{i|k}\to_{\Prod} p_{i+1|k}$ for all $i=1,\ldots N-1$ and $p_{k}\to_{\Prod} p_{1|k}$.  Here, $p_{i|k}$ is a notation used frequently in MPC, which denotes the $i$th state of the predicted trajectory at time $k$.  Moreover, we denote the set $\bP(p_{k},N)$
as the set of all finite trajectories of horizon $N$ from a state $p_{k}\in S_{\Prod}$.  Note that the finite predicted trajectory $\bp_{k}$ of $\Prod$ uniquely projects to a finite trajectory $\bq_{k}:=\gamma_{\T}(\bp_{k})$ of $\T$.

For the current state $q_{k}$ at time $k$, we denote the observed reward at any state $q\in Q$ as $R_{k}(q)$, and we have that
\be
R_{k}(q)=
\begin{cases}
\R(q, k)& \textrm{ if } q\in \N(q_{k},k)\\ 0 & \textrm{otherwise}.
\end{cases}
\ee
Note that $\R(q, k)=0$ if $q\notin \N(q_{k},k)$ because the rewards outside of the neighbourhood cannot be observed.  We can now define the \emph{predicted reward} associated with a predicted trajectory $\bp_{k}\in \bP(p_{k},N)$ at time $k$.  The predicted reward of $\bp_{k}$, denoted as $\Re_{k}(\bp_{k})$, is simply the amount of accumulated rewards by $\gamma_{\T}(\bp_{k})$ of $\T$:
\be
\label{def:predictedreward}
\Re_{k}(\bp_{k})=\sum_{i=1}^{N} R_{k}\left(\gamma_{\T}(p_{i|k})\right).
\ee


The receding horizon controller executed at the initial state at time $k=0$ is described next. This is a special case because the initial state of $\Prod$ is not unique, and as a result we can pick any initial state of $\mathcal P$ from the set $S_{\mathcal P0}=\{q_{0}\}\times S_{\mathcal B0}$.  We denote the controller executed at the initial state as $\rh^{0}(S_{\mathcal P0})$, and we define it as follows

\bea
\label{prob:initcontroller}
\bp_{0}^{\star}&=&\rh^{0}(S_{\mathcal P0})\nonumber\\&:=&\argmax_{\bp_{0} \in \{\bP(p_{0},N) \st V(p_{0})<\infty\}} \Re_{0}(\bp_{0}).
\eea

The controller maximizes the predicted cumulative rewards over all possible projected trajectories over horizon $N$ initiated from a state $p_{0}\in S_{\Prod 0}$ where the energy is finite, and returns the optimal projected trajectory $\bp_{0}^{\star}$.   The requirement that $V(p_{0})<\infty$ is critical because otherwise, the trajectory starting from $p_{0}$ cannot be accepting.  If there does not exist $p_{0}$ such that $V(p_{0})<\infty$, then an accepting trajectory does not exist and there is no trajectory of $\T$ satisfying the LTL formula (\ie Prob. \ref{problem:main} has no solution).

\begin{lemma}[Feasiblity of \eqref{prob:initcontroller}]
\label{lem:feasiblityInit}
Optimization problem \eqref{prob:initcontroller} always has at least one solution if there exists $p_{0}$ such that $V(p_{0})<\infty$.
\end{lemma}
\begin{proof}
The proof follows from the fact that $\T$ is non-blocking, and thus the set $\bP(p_{0},N)$ is not empty.
\end{proof}
%

Next, the receding horizon control algorithm for any time instant $k=1,2,\ldots$ and corresponding state $p_{k}\in S_{\Prod}$ is presented.  This controller is of the form
\be
\label{eq:rhcontroller}
\bp_{k}^{\star}=\rh(p_{k}, \bp^{\star}_{k-1})
\ee
\ie it depends both on the current state $p_{k}$ and the optimal predicted trajectory $\bp^{\star}_{k-1}=p^{\star}_{1|k-1}\ldots p^{\star}_{N|k-1}$ obtained at the previous time-step.    Note that, by the nature of a receding horizon control scheme, the first control of the previous predicted trajectory is always applied.  Therefore, we have the following equality
\be
\label{eq:recedingEq}
p_{k}=p^{\star}_{1|k-1}, k=1,2,\ldots .
\ee
As it will become clear in the text below, $\bp^{\star}_{k-1}$ is used to enforce repeated executions of this controller to eventually reduce the energy of the state on $\Prod$ to $0$.

We define controller \eqref{eq:rhcontroller} with the following three cases:
\subsubsection{Case 1. $V(p_{k})>0$ and $V(p^{\star}_{i|k-1})\neq 0$ for all $i=1,\ldots,N$} In this case, the receding horizon controller is defined as follows.
\bea
\label{rhcontrollercase1}
\bp^{\star}_{k}&=&\rh(p_{k}, \bp^{\star}_{k-1})\nonumber\\
&:=&\argmax_{\bp_{k} \in \bP(p_{k},N)} \Re_{k}(\bp_{k}), \nonumber\\
&&\textrm{subject to: } V(p_{N|k})< V(p^{\star}_{N|k-1}).
\eea

The key to guarantee that the energy of the states on $\Prod$ eventually decreases is the terminal constraint $V(p_{N|k})< V(p^{\star}_{N|k-1})$, \ie the optimal finite predicted trajectory $p^{\star}_{k}$ must end at a state with lower energy than that of the previous predicted trajectory $p^{\star}_{k-1}$.  This terminal constraint mechanism is graphically illustrated in Fig.~\ref{fig:RHcontrollerCase1and2}.
\begin{figure}[!ht]
	\centering
	\includegraphics[scale=.45]{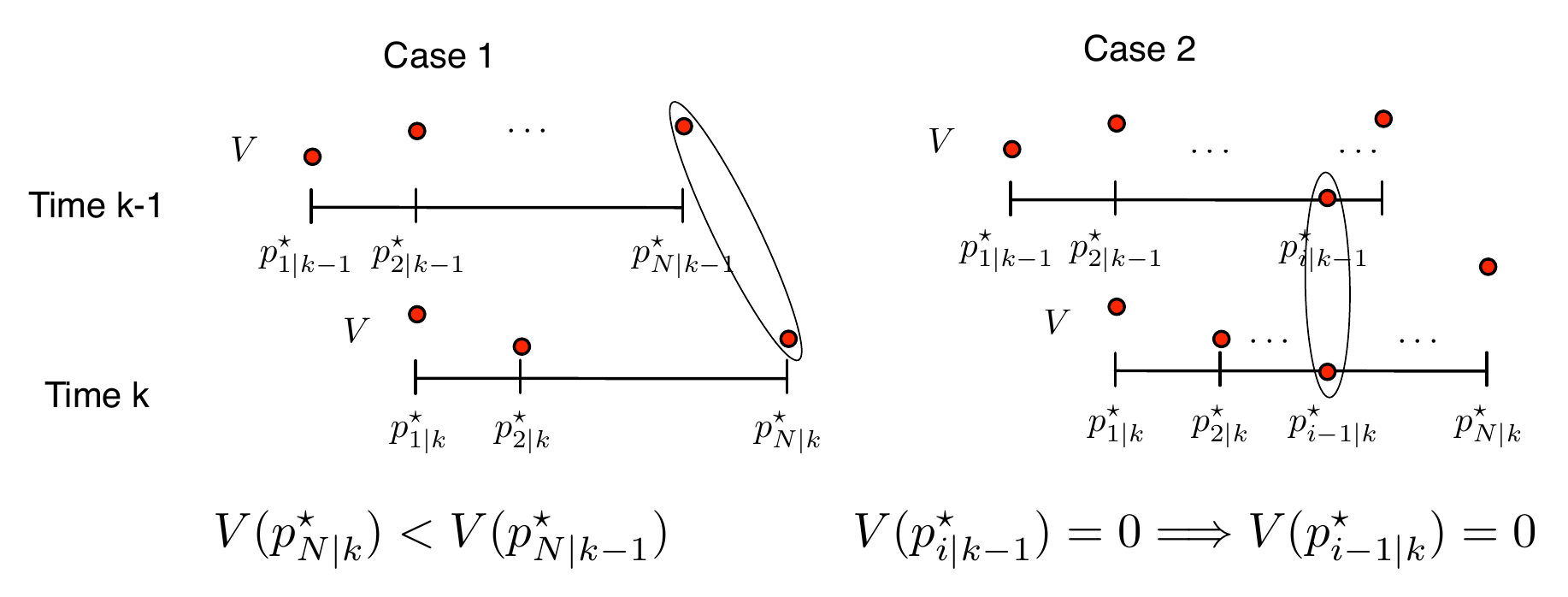}\caption{Constraints enforced for the receding horizon control law $\bp_{k}^{\star}=\rh(p_{k}, \bp^{\star}_{k-1})$ for Cases 1 and 2.  }
\label{fig:RHcontrollerCase1and2}
\end{figure}

To verify the feasibility of the optimization problem under this constraint, we make use of the third property of $V$ in Thm. \ref{thm:Vfacts}.  Namely, each state with positive finite energy can make a transition to a state with strictly lower energy.

\begin{lemma}[Feasibility of \eqref{rhcontrollercase1}]
\label{lem:feasibilityCase1}
Optimization problem \eqref{rhcontrollercase1} always has at least one solution if $V(p_{k})<\infty$.
\end{lemma}
\begin{proof}
Given $\bp^{\star}_{k-1}=p^{\star}_{1|k-1}\ldots p^{\star}_{N|k-1}$, since $p_{k}=p^{\star}_{1|k-1}$, we have $p_{k}\to_{\Prod} p^{\star}_{2|k-1}$.  Therefore, we can construct a finite predicted trajectory $\bp_{k}=p_{1|k}\ldots p_{N|k}$ where $p_{i|k}=p^{\star}_{i+1|k-1}$ for all $i=1,\ldots,N-1$.  Using Thm. \ref{thm:Vfacts} (iii), there exists a state $p$ where $p_{N-1|k}\to_{\Prod} p$ such that $V(p)<V(p_{N-1|k})$.  Setting $p_{N|k}=p$, the finite trajectory $\bp_{k}=p_{1|k}\ldots p_{N|k}\in \bP(p_{k},N)$ satisfies the constraint $V(p_{N|k})< V(p^{\star}_{N|k-1})$, and therefore \eqref{rhcontrollercase1} has at least one solution.
\end{proof}


\subsubsection{Case 2. $V(p_{k})>0$ and there exists $i\in \{1,\ldots,N\}$ with $V(p^{\star}_{i|k-1})=0$}
We denote $i^{0}(\bp^{\star}_{k-1})$ as the index of the first occurrence in $\bp^{\star}_{k-1}$ where the energy is $0$, \ie $V(p^{\star}_{i^{0}(\bp^{\star}_{k-1})|k-1})=0$.  We then propose the following controller.
\bea
\label{rhcontrollercase2}
\bp^{\star}_{k}&=&\rh(p_{k}, \bp^{\star}_{k-1})\nonumber\\
&:=&\argmax_{\bp_{k} \in \bP(p_{k},N)} \Re_{k}(\bp_{k}), \nonumber\\
&&\textrm{subject to: } V(p^{\star}_{i^{0}(\bp^{\star}_{k-1})-1|k})=0.
\eea

Namely, this controller enforces a state in the optimal predicted trajectory to have $0$ energy if the previous predicted trajectory contains such a state.   This constraint is illustrated in Fig. \ref{fig:RHcontrollerCase1and2}.  Note that, if $i^{0}(\bp^{\star}_{k-1})=1$, then from \eqref{eq:recedingEq}, the current state $p_{k}$ is such that $V(p_{k})=0$, and Case $2$ does not apply but Case $3$ (described below) applies instead.

\begin{lemma}[Feasibility of \eqref{rhcontrollercase2}]
\label{lem:feasibilityCase2}
Optimization problem \eqref{rhcontrollercase2} always has at least one solution if $V(p_{k})<\infty$.
\end{lemma}
\begin{proof}
Given $\bp^{\star}_{k-1}=p^{\star}_{1|k-1}\ldots p^{\star}_{N|k-1}$, since $p_{k}=p^{\star}_{1|k-1}$, we have $p_{k}\to_{\Prod} p^{\star}_{2|k-1}$.  Therefore, we can construct a finite predicted trajectory $\bp_{k}=p_{1|k}\ldots p_{N|k}$ where $p_{i|k}=p^{\star}_{i+1|k-1}$ for all $i=1,\ldots,N-1$.  If we let $p_{N|k}$ to be any state where $p_{N-1|k}\to_{\Prod} p_{N|k}$ and $V(p_{N|k})<\infty$, then $\bp_{k}=p_{1|k}\ldots p_{N|k}\in \bP(p_{k},N)$ satisfies the constraint.  Thm. \ref{thm:Vfacts} (iii) gurantees that such a state $p_{N|k}$ exists.
\end{proof}

\subsubsection{Case 3, $V(p_{k})=0$}
In this case, the terminal constraint is that energy value of the terminal state is finite.  The controller is defined as follows.
\bea
\label{rhcontrollercase3}
\bp^{\star}_{k}&=&\rh(p_{k}, \bp^{\star}_{k-1})\nonumber\\
&:=&\argmax_{\bp_{k} \in \bP(p_{k},N)} \Re_{k}(\bp_{k}).\nonumber\\
&&\textrm{subject to: } V(p_{N|k})<\infty.
\eea

\begin{lemma}[Feasiblity of \eqref{rhcontrollercase3}]
\label{lem:feasibilityCase3}
Optimization problem \eqref{rhcontrollercase3} always has at least one solution.
\end{lemma}
\begin{proof}
If $V(p_{k})=0$, then there exists $p_{1|k}$ such that $p_{k}\to_{\Prod}p_{1|k}$ and $V(p_{1|k}) < \infty$ (if not, then $V(p_{k})$ must equal to $\infty$).  From Thm. \ref{thm:Vfacts} (iii), we have that there exists $p_{2|k}$ such that $p_{1|k} \to_{\Prod} p_{2|k}$ and $V(p_{2|k}) < V(p_{1|k}) < \infty$.  By induction, there exists $\bp_{k}\in \bP(p_{k},N)$ such that $V(p_{N|k})<\infty$.
\end{proof}

\begin{remark}
The proposed receding horizon control law is designed using an extension of the terminal constraint approach in model predictive control \cite{rawlings2009} to finite deterministic systems. The particular setting of the B\"uchi acceptance condition, combined with the energy function $V$, makes it possible to obtain a non-conservative analogy of the terminal constraint approach, via either a terminal inequality condition \eqref{rhcontrollercase1} 
or a terminal equality condition \eqref{rhcontrollercase2}.
\end{remark}

\subsection{Control algorithm and its correctness}
\label{subsec:controlalg}
The overall control strategy for the transition system $\mathcal T$ is given in Alg. \ref{alg:main}.
After the off-line computation of the product automaton and the energy function, the algorithm applies the receding horizon controller $\rh^{0}(S_{\Prod 0})$ at time $k=0$, or $\rh(p_{k},\bp^{\star}_{k-1}))$ at time $k>0$.  At each iteration of the algorithm, the receding horizon controller returns the optimal predicted trajectory $\bp^{\star}_{k}$.  The immediate transition $(p_{k},p^{\star}_{1|k})$ is applied on $\Prod$ and the corresponding transition $(q_{k},\gamma_{\T}(p^{\star}_{1|k}))$ is applied on $\T$.  This process is then repeated at time $k+1$.

\begin{algorithm}
\caption{Receding horizon control algorithm for $\T=(Q,q_0,\Delta,\omega,\Pi,h)$, given an LTL formula $\phi$ over $\Pi$}
\begin{algorithmic}[1]\label{alg:main}
\small
\REQUIRE
\STATE Construct a B\"{u}chi automaton $\B = (S_{\B},S_{\B 0},2^{\Pi},\delta_\B,F_{\B})$ corresponding to $\phi$.
\STATE Construct the product automaton $\mathcal P=\mathcal T\times \mathcal B=(S_\Prod,S_{\Prod0},\Delta_\Prod, \omega_{\Prod},F_\Prod)$.  Find $V(p)$ for all $p\in S_{\Prod}$ \cite{ding2010receding}.
\end{algorithmic}
\begin{algorithmic}[1]
\ENSURE
\IF{there exists $p_{0}\in S_{\mathcal P0}$ such that $V(p_{0})\neq \infty$}
\STATE Set $k=0$.
\STATE Observe rewards for all $q\in \N(q_{0},k)$ and obtain $R_{0}(q)$.
\STATE Obtain $\bp^{\star}_{0}=\rh^{0}(S_{\Prod 0})$.
\STATE Implement transition $(p_{0}, p^{\star}_{1|0})$ on $\Prod$ and transition $(q_{0},\gamma_{\T}(p^{\star}_{1|0}))$ on $\T$.
\STATE Set $k=1$
\LOOP
\STATE Observe rewards for all $q\in \N(q_{k},k)$ and obtain $R_{k}(q)$.
\STATE Obtain $\bp^{\star}_{k}=\rh(p_{k},\bp^{\star}_{k-1})$.
\STATE Implement transition $(p_{k}, p^{\star}_{1|k})$ on $\Prod$ and transition $(q_{k},\gamma_{\T}(p^{\star}_{1|k}))$ on $\T$.
\STATE Set $k\leftarrow k+1$
\ENDLOOP
\ELSE
\STATE There is no run originating from $q_{0}$ that satisfies $\phi$.
\ENDIF
\end{algorithmic}
\end{algorithm}

First, we show that the receding horizon controllers used in Alg. \ref{alg:main} are always feasible.  We use a recursive argument, which shows that if the problem is feasible for the initial state, or at time $k=0$, then it remains feasible for all future time-steps $k=1,2,\ldots$.
\begin{theorem}[Recursive Feasiblity]
\label{thm:feasiblityAlg}
 If there exists $p_{0}\in S_{\mathcal P0}$ such that $V(p_{0})\neq \infty$, then $\rh^{0}(S_{\Prod 0})$ is feasible and $\rh(p_{k},\bp^{\star}_{k-1}))$ is feasible for all $k=1,2,\ldots$.
\end{theorem}
\begin{proof}
From Lemma \ref{lem:feasiblityInit}, $\rh^{0}(S_{\Prod 0})$ is feasible.   From the definition of $V(p)$, for all $p\in S_{\Prod}$, if $p\to_{\Prod}p'$, then $V(p')<\infty$ if and only if $V(p)<\infty$.  Since $\rh^{0}(S_{\Prod 0})$ is feasible, we have $p_{1}=p^{\star}_{1|0}$ and thus $V(p_{1})<\infty$.  At each time $k>0$, if $V(p_{k})<\infty$, from Lemmas \ref{lem:feasibilityCase1}, \ref{lem:feasibilityCase2} and \ref{lem:feasibilityCase3}, we have that controller $\rh(p_{k},\bp^{\star}_{k-1})$ is feasible.  Since $p_{k+1}=p^{\star}_{1|k}$, we have $V(p_{k+1})<\infty$.  Using induction we have that $\rh(p_{k},\bp^{\star}_{k-1})$ is feasible for all $k=1,2,\ldots$.
\end{proof}
Finally, we show that Alg. \ref{alg:main} always produces an infinite trajectory satisfying the given LTL formula $\phi$, giving a solution to Prob. \ref{problem:main}.

\begin{theorem}[Correctness of Alg. \ref{alg:main}]
Assume that there exists a satisfying run originating from $q_{0}$ for a transition system $\mathcal T$ and an LTL formula $\phi$. Then, Alg. \ref{alg:main} produces an (infinite) trajectory $\bq=q_{0}q_{1}\ldots$ satisfying $\phi$.
\end{theorem}
\begin{proof}
If there exits a satisfying run originating from $q_{0}$, then there exists a state $p_{0}\in S_{\Prod 0}$ such that $V(p_{0})<\infty$.  Therefore, from Thm. \ref{thm:feasiblityAlg}, the receding horizon controller is feasible for all $k>0$, and Alg. \ref{alg:main} will always produce an infinite trajectory $\bq$.

At each state $p_{k}$ at time $k>0$, if $V(p_{k})>0$, then either Case 1 or Case 2 of the controller $\rh(p_{k})$ applies.  If Case 1 applies, since $V(p^{\star}_{k|N})>V(p^{\star}_{k+1|N}>V(p^{\star}_{k+2|N})\ldots$, there exists $j>k$ such that $V(p^{\star}_{j|N})=0$.  This is because the state-space $S_{\Prod}$ is finite, and therefore, there is only a finite number of possible values for the energy function $V(p)$.  At time $j$, Case 2 of the proposed controller becomes active until time $l=j+i^{0}(p^{\star}_{j})$, where $V(p_{l})=0$. Therefore, for each time $k$, if $V(p_{k})>0$, there exists $l>k$ such that $V(p_{l})=0$ by repeatedly applying the receding horizon controller.   If $V(p_{k})=0$, then Case 3 of the proposed controller applies, in which case either $V(p_{k+1})=0$ or $V(p_{k+1})>0$.  In either case, using the previous argument, there exists $j>k$ where $V(p_{j})=0$.

Therefore, at any time $k$, there exists $j>k$ where $V(p_{j})=0$.  Furthermore, since $j$ is finite, we can conclude that the number of times where $V(p_{k})=0$ is infinite.  By the definition of $V(p), p\in S_{\Prod}$, $V_{k}=0$ is equivalent to that $p_{k}\in F^{\star}_{\Prod}\subseteq F_{\Prod}$.  Therefore, the trajectory $\bp$ is accepting.  The trajectory produced on $\T$ is exactly the projection $\bq=\gamma_{\T}(\bp)$, and thus, it can be concluded that $\bq$ satisfies $\phi$, which completes the proof.
\end{proof}

\subsection{Discussions}
It is possible to extend the optimization problem of maximizing rewards to other meaningful cost functions.  For example, it is possible to assign penalties or costs on states of the system and minimize the accumulated cost of trajectories in the horizon.   It is also possible to define costs on state transitions and minimize the control effort (or the combination of this cost function with the one above).

The complexity of the off-line portion of Alg. \ref{alg:main} depends on the size of $\Prod$.  Denoting $|S |$ as the cardinality of a set $S$, from \cite{gastin2001fast}, a B\"{u}chi automaton translated from an LTL formula over $\Pi$ contains at most $|\Pi|\times 2^{|\Pi|}$ states\footnote{In practice, this upper limit is almost never reached (see \cite{KB-TAC08-LTLCon}).}.  Therefore, the size of $S_\Prod$ is bounded by $|Q|\times |\Pi|\times 2^{|\Pi|}$.  From \cite{ding2010receding}, the complexity of generating the energy function is $O(|S_{\Prod}|^{2}+|F_{\Prod}|^{3})$.  The complexity of the on-line portion of Alg. \ref{alg:main} is highly dependent on the horizon $N$.  If the maximal number of transitions at each state of $\Prod$ is $\Delta^{\mathtt{max}}_{\Prod}$, then the complexity at each iteration of the receding horizon controller is bounded by $(\Delta^{\mathtt{max}}_{\Prod})^{N}$, assuming a depth first search algorithm is used to find the optimal trajectory.   It may be possible to reduce this complexity from exponential to polynomial if one applies a more efficient graph search algorithm using Dynamic Programming.  This will be studied in future research.

\section{Software Implementation and Case Study}
\label{sec:example}
The control framework presented in this paper was implemented in a user friendly software package, available on \url{http://hyness.bu.edu/LTL_MPC.html}.  To utilize this software, a user needs to input the finite transition system $\T$, an LTL formula $\phi$, the horizon $N$, and a function $\R(q,k)$ that generates the time-varying rewards defined on the states of $\T$.  The software executes the control algorithm outlined in Alg. \ref{alg:main}, and produces a trajectory in $\T$ that satisfies $\phi$ and maximizes the rewards collected locally with the proposed receding horizon control laws.  This software uses the LTL2BA \cite{gastin2001fast} tool for the translation of an LTL formula to a B\"uchi automaton.

We now present a case study applying the software package.   In this case study, we use the transition system defined as vertices of a rectangular grid as shown in Fig. \ref{fig:Texample}.  We consider the following LTL formula, which expresses a robotic surveillance task:
\bea
\label{eq:LTLformulaex}
 \phi&:=&\Always \Event \mathtt{base} \n
&&\Land \Always(\mathtt{base}\longrightarrow \Next\neg \mathtt{base} \Until \mathtt{survey})\n
&&\Land \Always(\mathtt{survey}\longrightarrow \Next\neg \mathtt{survey} \Until \mathtt{recharge})\n
&&\Land \Always \neg \mathtt{unsafe}.
\eea

The first line of $\phi$, $\Always \Event \mathtt{base}$, enforces that the state with observation $\mathtt{base}$ is repeatedly visited (possibly for uploading data).   The second line ensures that after $\mathtt{base}$ is reached, the system is driven to a state with observation $\mathtt{survey}$, before going back to $\mathtt{base}$.  Similarly, the third line ensures that after reaching $\mathtt{survey}$, the system is driven to a state with observation $\mathtt{recharge}$, before going back to $\mathtt{survey}$.  The last line ensures that, at any time, the states with observation $\mathtt{unsafe}$ should be avoided.

We assume that at each state $q\in Q$, the rewards at state $q'$ can be observed if the Euclidean distance between $q$ and $q'$ is less than or equal to $25$.  In this case study, we define $\R(q,k)$ as follows. At time $k=0$, the reward value $\R(q,0)$ at each state $q$ is generated randomly by a uniform sampling in the range of $[10,25]$.  At each subsequent time $k>0$, if the reward value at a state is positive, then it decays with a specific rate.  Otherwise, there is a probability that a reward is assigned to this state with a value chosen by a uniform sampling in the range of $[10,25]$.   In this case study, the states with rewards can be seen as ``targets'', and the reward values can be seen as the ``amount of interest'' associated with each target.  The control objective of maximizing the collected rewards can be interpreted as maximizing the information gathered from surveying states with high interest.

\begin{figure}[!ht]
	\centering
	\label{fig:seq}
	\subfloat[]{	\includegraphics[scale=.4]{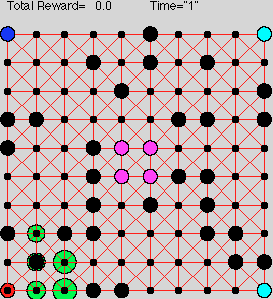}}
	\subfloat[]{	\includegraphics[scale=.4]{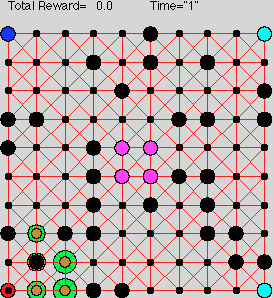}}\\
	\subfloat[]{	\includegraphics[scale=.4]{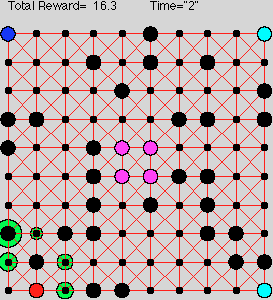}}
	\subfloat[]{	\includegraphics[scale=.4]{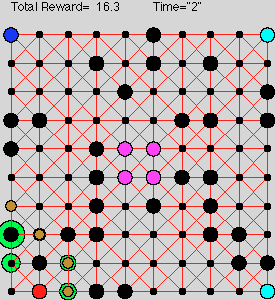}}
	\caption{Snapshots of the system trajectory under the proposed receding horizon control laws.  In all snapshots, the state with rewards are marked in green, where the size of the state is proportional with the associated reward.  {\bf a)} At time $k=0$, the initial state of the system is marked in red (in the lower left corner).  {\bf b)} The controller $\bp^{\star}_{0}=\rh^{0}(S_{\Prod 0})$ is computed at the initial state.  The optimal predicted trajectory $\bp^{\star}_{0}$ is marked by a sequence of states in brown.  {\bf c)} The first transition $q_{0}\to_{\T}q_{1}$ is applied on $\T$ and transition $p_{0}\to_{\Prod}p_{1}$ is applied on $\Prod$.  The current state ($q_{1}$) of the system is marked in red.  {\bf d)} The controller $\bp^{\star}_{1}=\rh(p_{1},\bp^{\star}_{0})$ is computed at $p_{1}$.  The optimal predicted trajectory $\bp^{\star}_{1}$ is marked by a sequence of states in brown.}
\label{fig:snapshots}
\end{figure}

By applying the method described in the paper, our software package first translates $\phi$ to a B\"uchi automaton $\B$, which has 12 states.  This procedure took $0.5$ second on a Macbook Pro with a 2.2GHz Quad-core CPU.  Since $\T$ contains $100$ states, $|S_\Prod|$ is $1200$.  The generation of the product automaton $\Prod$ and the computation of the energy function $V$ took $4$ seconds.   In this case study, we chose the horizon $N$ to be $4$.  By applying Alg. \ref{alg:main}, some snapshots of the system trajectory are shown in Fig. \ref{fig:snapshots}.   Each iteration of Alg. \ref{alg:main} took $1-3$ seconds (due to different numbers of graph searches needed, the computation time varies for each iteration).

We applies the control algorithm for $100$ time-steps.  We plotted the results after $100$ time-steps in Fig. \ref{fig:JandRR}.  At the top, we plot the energy $V(p)$ at the each time-step.  We see that after 55 time-steps, the energy is $0$, meaning that an accepting state is reached.    Note that, each time an accepting state is reached, the system visits the $\mathtt{base}$, $\mathtt{survey}$ and $\mathtt{recharge}$ states at least once \ie one cycle of the surveillance mission task ($\mathtt{base}$ -- $\mathtt{survey}$ -- $\mathtt{recharge}$) is completed.  We also compare the receding horizon controller with the controller proposed in \cite{ding2010receding} at the bottom of Fig. \ref{fig:JandRR}.  We clearly see that the receding horizon controller proposed in this paper performs better in terms of rewards collection, since it reacts much quicker to the time varying rewards.  An example video of the evolution of the system trajectory is also available at \url{http://hyness.bu.edu/LTL_MPC.html}.  

\begin{figure}[!ht]
	\centering
	\includegraphics[scale=.4]{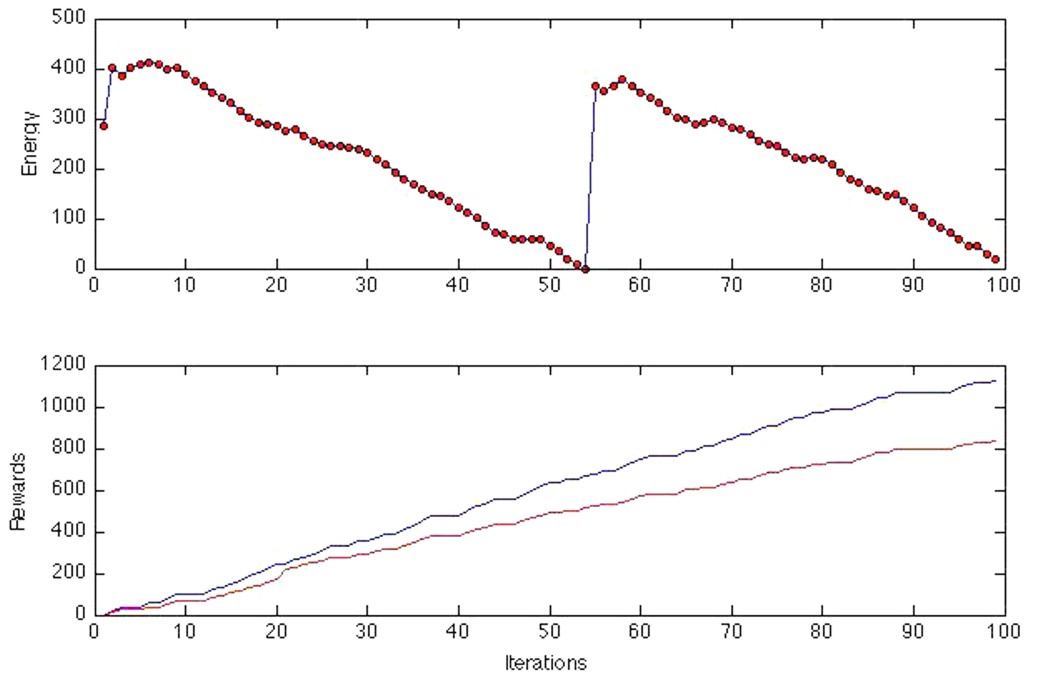}\caption{Upper figure: plot of energy $V(p)$ at the current state for $100$ time-steps.  Bottom figure: in blue, plot of the cumulative rewards collected in $100$ time-steps by the proposed receding horizon controller; in red, plot of the cumulative rewards collected by the controller in \cite{ding2010receding} using the same reward function $\R(q,k)$. }
\label{fig:JandRR}
\end{figure}




\section{Conclusion and final remarks}
\label{sec:conclusion}


In this paper, a receding horizon control framework that optimizes the trajectory of a finite deterministic system locally, while guaranteeing that the infinite trajectory satisfies a given linear temporal logic formula, was proposed. The optimization criterion was defined as maximization of time-varying rewards associated with the states of the system. A control strategy that makes real-time control decisions in terms of maximizing the reward while ensuring satisfaction of the LTL specification was developed. The proposed framework is a step toward synergy of model predictive control and formal controller synthesis, which is beneficial for both areas.

Future research deals with the extension of the proposed framework to finite probabilistic systems, such as Markov decision processes or partially observed Markov decision processes,  where the specifications are given as formulas of probabilistic temporal logic.


\bibliographystyle{IEEEtran}
\bibliography{Papers}

\end{document}